\documentclass{amsart}
\usepackage{graphicx}
\usepackage{amsmath}

\newtheorem{thm}{Theorem}[section]

\newtheorem{lem}[thm]{Lemma}
\newtheorem{cor}[thm]{Corollary}

\theoremstyle{definition}
\newtheorem{definition}[thm]{Definition}

\theoremstyle{remark}

\numberwithin{equation}{section}

\renewcommand{\char}{\textnormal{char}}  
\renewcommand{\Vec}{\textbf{Vec}}
\newcommand{\sVec}{\textbf{sVec}}
\newcommand{\id}{\textnormal{id}}
\renewcommand{\Im}{\textnormal{Im}}
\newcommand{\Ker}{\textnormal{Ker}}
\newcommand{\Z}{\textnormal{Z}}
\newcommand{\ad}{\textnormal{ad}}
\newcommand{\gr}{\mathfrak{gr}}

\usepackage{hyperref}
\hypersetup{
	colorlinks=true,
	linkcolor=black,
	urlcolor=blue,
}

\title{Superalgebra in characteristic 2}

\author{Aaron Kaufer}

\begin{document}
	
	\maketitle
	
	\begin{abstract}
		
		Following the work of Siddharth Venkatesh, we study the category $\sVec_2$. This category is a proposed candidate for the category of supervector spaces over fields of characteristic $2$ (as the ordinary notion of a supervector space does not make sense in charcacteristic $2$). In particular, we study commutative algebras in $\sVec_2$, known as $d$-algebras, which are ordinary associative algebras $A$ together with a linear derivation $d:A \to A$ satisfying the twisted commutativity rule: $ab = ba + d(b)d(a)$. In this paper, we generalize many results from standard commutative algebra to the setting of $d$-algebras; most notably, we give two proofs of the statement that Artinian $d$-algebras may be decomposed as a direct product of local $d$-algebras. In addition, we show that there exists no noncommutative $d$-algebras of dimension $\leq 7$, and that up to isomorphism there exists exactly one $d$-algebra of dimension $7$. Finally, we give the notion of a Lie algebra in the category $\sVec_2$, and we state and prove the Poincare-Birkhoff-Witt theorem for this category.

	\end{abstract}
	
	\tableofcontents
	
	\newpage
	
	\section{Introduction}
	
	The concept of superalgebra finds its origins in supersymmetry, a theory from particle physics that attempts to explain the behaviors of elementary particles such as bosons and fermions. Supersymmetry has applications, in particular, to both string theory and quantum mechanics.
	
	In supersymmetry, many objects arise as natural analogs of standard algebraic objects. For example, the analog of a vector space over a field $F$ is a vector superspace, which is a vector space $V$ over $F$ that can be decomposed as $V=V_0 \oplus V_1$ (the ``even" and ``odd" components of $V$). If $\char(F) = 2$, however, the notion of a vector superspace over $F$ makes no sense since the concept of parity is nonexistent. 
	
	When $\char(F) \neq 2$, the category $\Vec$ of vector spaces over $F$ and the category $\sVec$ of vector superspaces over $F$ both naturally form symmetric tensor categories. This allows for various algebraic structures, such as commutative algbras and Lie algebras, to be defined within these categories. In $\Vec$, the structures are just standard commutative algberas and Lie algebras, whereas in $\sVec$, the structures are commutative superalgberas and Lie superalgebras, respectively. The study of algebraic structures in the category $\sVec$ is known as superalgebra. Necessarily, the characteristic of the base field must not be $2$ in order for superalgebra to make sense.
	
	In [5], Siddharth Venkatesh proposes an candidate for the notion of a vector superspace over a field of characteristic $2$. In particular, if $\char(F)=2$, then he contructs the category $\sVec_2$ as the category of representations of a given Hopf algebra (namely, the Hopf algbera $D=F[d]/(d^2)$ with primitive $d$). 
	
	Concretely, we have the following definition:
	
	\begin{definition}
		Suppose $F$ is an algebraically closed field of characteristic $2$. Then, the category $\sVec_2$ is the category whose objects are vector spaces $V$ over $F$ equipped with a a linear map $d=d_V: V \to V$, called the differential, such that $d^2=0$. The morphisms of this category are linear maps which commute with the differential. (That is, a linear map $T : V \to W$ such that $T \circ d_V = d_W \circ T$.
		
		The category $\sVec_2$ can be made into a tensor category by setting the tensor product $\otimes$ to be the normal tensor product, and defining 
		$$d_{V \otimes W} = d_V \otimes 1_W + 1_V \otimes d_W$$
		where $1_V$ and $1_W$ are just the identity maps of $V$ and $W$ respectively. 
		
		Furthermore, $\sVec_2$ can be made into a symmetric tensor categroy by defining the commutativity map $c_{V,W}: V \otimes W \to W \otimes V$ by:
		$$ c(v \otimes w) = w \otimes v + d(w) \otimes d(v).$$
	\end{definition}
	
	Since $\sVec_2$ is a symmetric tensor category, we can define algebraic structures within it, which are, as Venkatesh suggests, essentially analogues of superalgebraic structures in characteristic $2$. In this paper, we define commutative algebras and Lie algebras in $\sVec_2$, and we generalize results from standard commutative and Lie algebra to this category.
	
	\section{Commutative Algebras in $\sVec_2$}
	
	\subsection{Definition}

	In this section, we give the definition of a commutative algebra in $\sVec_2$. In addition, we generalize results from standard commutative algebra, and we classify finite dimensional commutative algebras in $\sVec_2$ up to dimension $7$. We recall that the base field $F$ is assumed to be algebraically closed and have characteristic $2$.
	
	As in [3], a commutative algebra in $\sVec_2$ is an object $A \in \sVec_2$ together with a morphism $m: A \otimes A \to A$ which satisfies associativity:
	$$ m \circ (\id \otimes m) = m \circ (m \otimes \id)$$
	and commutativity:
	$$ m = m \circ c$$
	where here $c=c_{A,A}:A \otimes A \to A \otimes A$ is the commutativity map. 
	
	If we write $m(a \otimes b)$ as $ab$ for $a,b \in A$, and we recall that the commutativity map is defined by $c(a \otimes b) = b \otimes a + d(b) \otimes d(a)$, then we can rewrite the first axiom as the familiar associativity axiom:
	$$ (ab)c = a(bc)$$
	and the commutativity axiom becomes:
	$$ ab = ba + d(b)d(a).$$
	
	Furthermore, $d_{A \otimes A} = d \otimes 1 + 1 \otimes d$, so if $a,b \in A$, then $d(a \otimes b) = d(a) \otimes b + a \otimes d(b)$. Since multiplication is a morphism in $\sVec_2$, it must commute with the differential $d$, so we have $d(ab) = d(a)b + ad(b)$ for all $a,b \in A$.
	
	Then, as we can see, a commutative algebra in $\sVec_2$ is just an ordinary associative algebra $A$ over $F$ togther with a linear derivation $d:A \to A$ such that $ab=ba+d(b)d(a)$ for all $A$. 
	
	To avoid confusion, we will use the term ``commutative" to refer to the standard condition $ab=ba$, and we will use the term  ``$d$-commutative" to refer the condition $ab=ba+d(b)d(a)$. An algebra that is $d$-commutative is called a $d$-algebra. 
	
	\subsection{General Facts and Constructions}
	
	For this section, we assume that $A$ is a $d$-algebra over $F$, where $\char(F)=2$ and $F$ is algebraically closed. Much of the theory of $d$-algebras comes from studying two important subalgebras of $A$, in particular $\Ker(d)$ and $\Im(d)$.
	
	\begin{lem}
		Suppose $a \in A$. Then $d(a)^2=0$.
	\end{lem}
	\begin{proof}
		This follows from $d$-commutativity: $0 = a \cdot a - a \cdot a = d(a)d(a)$.
	\end{proof}
	
	\begin{lem}
		$\Im(d) \subset \Ker(d) \subseteq \Z(A)$ where $\Z(A)$ is the center of $A$. Furthermore, they are all subalgebras of $A$.
	\end{lem}
	\begin{proof}
		All three sets are clearly subalgebras. $\Ker(d) \subseteq \Z(A)$ follows immediately from $d$-commutativity.
		
		$\Im(d) \subseteq \Ker(d)$ because $d^2=0$, and $\Im(d) \neq \Ker(d)$ because $1 \in \Ker(d)$ but $1^2 \neq 0$, so $1 \not \in \Im(d)$ by the previous lemma.
	\end{proof}
	
	\begin{lem}
		If $A$ is finite dimensional, then $\dim(\Im(d)) < \dim(\Ker(d))$.
	\end{lem}
	
	\begin{thm}
		\label{dim im >= 2}
		If $\dim(\Im(d)) \leq 2$ then $A$ is commutative.
	\end{thm}
	\begin{proof}
		We prove this in contrapositive form. Suppose that $A$ is noncommutative, so that there exists $a,b \in A$ such that $0 \neq [a,b] = d(a)d(b)$. Then, $d(a)d(b)=d(d(a)b) \in \Im(d)$, and the set $\{ d(a), d(b), d(a)d(b)\}$ is linearly independent over $F$, which proves that $\dim(\Im(d)) \geq 3$.
		
		To see that the given set in linearly independent, suppose $\alpha d(a) + \beta d(b) + \gamma d(a)d(b) = 0$ for $\alpha, \beta, \gamma \in F$. Then, mutliplying by $d(b)$ gives $\alpha d(a)d(b)=0$, so $\alpha=0$. Likewise, multiplying by $d(a)$ shows $\beta=0$, leaving us with $\gamma d(a)d(b)=0$. Thus, $\gamma=0$, so we see that $\{d(a), d(b), d(a)d(b)\}$ is linearly independent over $F$.
	\end{proof}
	
	\begin{cor}
		If $\dim(A) \leq 6$, then $A$ is commutative.
	\end{cor}
	\begin{proof}
		Because $\dim(\Im(d))+\dim(\Ker(d))=\dim(A) \leq 6$ and $\dim(\Im(d)) < \dim(\Ker(d))$, we see that $\dim(\Im(d)) < \frac{1}{2}\dim(A) \leq 3$. Thus, by the previous theorem, $A$ is commutative.
	\end{proof}
	
	After this corollary, it is natural to ask for the smallest example of a noncommutative $d$-commutative algebra. As it turns out $\dim(A)=7$ is the smallest example of such a $d$-algebra, and in order to construct an explicit example, we first develop the notion of a polynomial algebra in the category $\sVec_2$ in a manner that mirrors polynomial algebras over commutative rings. In particular, just as $F[x_1, \ldots, x_n]$ is a free commutative algebra generated by $\{ x_1, \ldots, x_n\}$, we wish to create some sort of free $d$-algebra generated by a finite set. The motivation for this definition is that we wish to adjoin $r$ $d$-commutative indeterminates $x_1,\ldots,x_r$ such that $d(x_i) = \xi_i \neq 0 $, and in addition we wish to adjoin $s$ indeterminates $y_1,\ldots,y_s$ such that $d(y_i)=0$ but $y_i \not \in \Im(d)$.
	
	\begin{definition}
		\label{Polynomial d-algebra def}
		The polynomial $d$-algbera, generated by $(r,s)$ indeterminates, denoted $P^r_s$, is defined as:
		$$ P_s^r := F[y_1,\ldots,y_s,\xi_1,\ldots,\xi_r] \langle x_1, \ldots, x_r\rangle \big / I$$
		where $I$ is the ideal generated by elements of the form $x_ix_j-x_jx_i-\xi_i \xi_j$ for $1 \leq i,j \leq r$. Here $R\langle x_1, \ldots, x_r\rangle$ denotes the free noncommutative algebra over $R$ generated by $\{ x_1, \ldots, x_r\}$.
	\end{definition}
	
	\begin{definition}
		Suppose $P=P_s^r$. Then we define the map $d=d_P:P \to P$ as the linear map such $d(x_i)= \xi_i$, $d(\xi_i)=0$, $d(y_i) = 0$, and we extend $d$ to products of elements by $d(ab)=ad(b)+d(a)b$.
	\end{definition}
	\begin{thm}
		Suppose $P=P_s^r$. The map $d=d_P$ is well defined and gives $P$ the structure of a $d$-algebra.
	\end{thm}
	\begin{proof}
		Set $R= F[y_1,\ldots,y_s,\xi_1,\ldots,\xi_r] \langle x_1, \ldots, x_r\rangle$. Since $d$ is clearly well defined on $R$, it suffices to show that the ideal $I$ is closed under $d$. For this, we note:
		$$d(x_ix_j-x_jx_i-\xi_i\xi_j) = d(x_ix_j)-d(x_jx_i)-d(\xi_i\xi_j) = x_i\xi_j+\xi_ix_j-x_j\xi_i-\xi_jx_i = 0.$$
		Thus, let $\alpha_{i,j}=x_ix_j-x_jx_i-\xi_i\xi_j$. Then for any product $a \cdot \alpha_{i,j} \cdot b$ with $a,b \in R$, we have $$d(a \cdot  \alpha_{i,j} \cdot b)=d(a) \cdot \alpha_{i,j} \cdot b + a \cdot \alpha_{i,j}  \cdot d(b).$$ Since any element in $I$ is a finite sum of elements of the form $a \cdot \alpha_{i,j} \cdot b$, this shows that $I$ is closed under $d$, so $d$ induces a well defined map on $R/I$.
		
		To show that $d$ makes $P$ into a $d$-algebra, we must verify $d$-commutativity holds, and by linearity, it suffices to show that it holds for monomials of the form $a_1\ldots, a_k$ and $b_1 \ldots b_l$ with $a_i,b_i \in \{ x_1,\ldots, x_r\}$. We do induction on $l+k$. The case of $l+k=2$ follows directly from the definition of $P$. Now, consider:
		\begin{align*}
		(a_1a_2\cdots a_k) (b_1 \cdots b_l) &= a_1 \left( b_1 \cdots b_la_2\cdots a_k + d(b_1 \cdots b_l)d(a_2\cdots a_k) \right) \\
		&= a_1b_1\cdots b_la_2 \cdots a_k + a_1d(b_1 \cdots b_l)d(a_2\cdots a_k) \\
		&= \left(b_1 \cdots b_l a_1 + d(a_1)d(b_1\ldots b_l)\right)a_2 \cdots a_k + a_1d(b_1 \cdots b_l)d(a_2\cdots a_k) \\
		&= (b_1 \cdots b_l)(a_1a_2\cdots a_k) + d(a_1 \cdots a_k)d(b_1 \cdots b_l).
		\end{align*}
		Here we used the induction hypothesis on the first and third lines. 
	\end{proof}
	
	
	Now that we have developed the notion of a polynomial $d$-algebra, we may easily construct examples of noncommutative $d$-algebras. In particular, consider the algebra:
	$$ A = P^2_0 \big / \left( x_1^2, \ x_1x_2,\  x_2^2, \ \xi_1x_1,\  \xi_2x_2,\  \xi_1x_2-\xi_2x_1 \right).$$
	This algebra has a basis $\{ 1, \xi_1, \xi_2, \xi_1\xi_2, x_1, x_2, \xi_1 x_2\}$, so we have $\dim(A)=7$. Furthermore, $A$ is noncommutative because $x_1x_2=0$, but $x_2x_1=\xi_1\xi_2$. As we will see in Section 6, it turns out that this $d$-algebra is the only noncommutative $d$-algebra of dimension $7$, up to isomorphism.
	
	The definition of a $d$-algebra homomorphism is as expected:
	
	\begin{definition}
		Let $A$ and $A'$ be $d$-algebras with differentials $d$ and $d'$, respectively. Then a map $\phi : A \to A'$ is a $d$-algebra homomorphism if it is an $F$-algebra homomorphism and it commutes with the differentials. That is, $\phi(d(a))=d'(\phi(a))$ for all $a \in A$. (Or in other words, $\phi$ is an $F$-algebra homomorphism as well as a morphism in $\sVec_2$).
	\end{definition}
	
	From this definition, it is easy to see that the kernel of a $d$-algebra homormorphism must be closed under $d$, so this gives us the natural definition of an ideal in the category $\sVec_2$:
	
	\begin{definition}
		Let $I$ be a left (resp. right) ideal in $A$. Then $I$ is a left (resp. right) $d$-ideal if it is closed under $d$.
	\end{definition}
	
	Although we initially differentiate between left and right ideals because $A$ need not be commutative, we will see later all $d$-ideals are necessarily two sided, so the distinction is not necessary. Furthermore, if $I$ is two sided and closed under $d$, then we may form the quotient algebra $A/I$, and the map $d$ naturally induces a differential $d_{A/I}$ on $A/I$, making $A/I$ a $d$-algebra.
	
	Part of the importance of polynommial algebras in standard commutative algebra is that any finitely generated commutative algebra is isomorphic to a quotient of a polynomial algebra. An analogous statement holds in the category $\sVec_2$:
	
	\begin{thm}
		Suppose $A$ is a finitely generated $d$-algebra. Then $A$ is isomorphic to a quotient of the polynomial $d$-algebra $P_s^r$ for some $r,s$. 
	\end{thm}
	\begin{proof}
		Since $A$ is finitely generated, we pay pick a generating set $\{a_1,a_2,\ldots,a_r, b_1,b_2,\ldots, b_s\}$ where we have arranged the generators so that $d(a_i) \neq 0$ and $d(b_i)=0$. Then, construct the homormorphism $\phi: P_s^r \to A$ by setting $\phi(x_i)=a_i$, $\phi(\xi_i)=d(a_i)$, and $\phi(y_i)=b_i$. It is trivially clear that $\phi$ is a surjective $d$-algebra homomorphism, so we have $A \cong P_s^r / \Ker(\phi)$.
	\end{proof}

	
	
	\subsection{Properties of $d$-Ideals}
	
	We devote this section to proving properties of $d$-ideals and showing that many of the standard commutative algebra theorems apply.
	
	\begin{thm}
		If $I$ is a left (resp. right) $d$-ideal, then it is also a right (resp. left) $d$-ideal. Thus, all $d$-ideals are two-sided.
	\end{thm}
	\begin{proof}
		Suppose $I$ is a left $d$-ideal. Let $a \in I$ and $r \in A$. Then $ra \in I$, so by $d$-commutativity,
		$$ ar = ra+d(r)d(a).$$
		Since $I$ is closed under $d$, we see $d(a)\in I$ and consequently $ra+d(r)d(a) \in I$. Thus $ar \in I$, hence $I$ is a right $d$-ideal. The proof that right $d$-ideals are also left $d$-ideals is completely analogous.
	\end{proof}
	
	In particular, as we mentioned earlier, there is an induced differential $d_{A/I}$ on $A/I$ that gives it the structure of a $d$-algebra.
	
	\begin{thm}
		Let $M$ be a maximal left (resp. right) ideal of $A$. Then $\Im(d) \subseteq M$.
	\end{thm}
	\begin{proof}
		Suppose for contradiction that $d(a) \not \in M$ for some $a \in A$. Then, the ideal $Ad(a)+M$ contains the element $d(a)$, so $Ad(a)+M \supset M$, so $Ad(a)+M = A$. Thus, there exists $r \in A$ and $m \in M$ such that $rd(a)+m=1$. Multiplying on the right by $d(a)$ gives $md(a)=d(a)$. But $d(a) \in Z(A)$, so $d(a)=d(a)m \in M$.
	\end{proof}
	\begin{cor}
		All maximal ideals are $d$-ideals. In particular, they are all two sided.
	\end{cor}	
	\begin{thm}
		Suppose $A$ is finitely generated, and $M$ is a maximal ideal. Then $A/M \cong F$.
	\end{thm}
	\begin{proof}
		To show that $A/M$ is a field, it suffices to show that it is commutative. To see this, note that for any $a,b \in A$, we have $ab-ba=d(a)d(b) \in M$ because $\Im(d) \subseteq M$. Then, because $F$ is algebraically closed	and $A/M$ is finitely generated, it follows that $A/M \cong F$ by Zariski's lemma.
	\end{proof}
	\begin{cor}
		If $A$ is finite dimensional and $M$ is a maximal ideal, then $\dim(M)=\dim(A)-1$.
	\end{cor}
	
	
	
	
	
	Many of the regular properties of "ideal arithmetic" carry over nicely to $d$-algebras.
	
	\begin{thm}
		\label{coprime intersect}
		Let $I$ and $J$ be coprime d-ideals (i.e. $I+J=A$). Then $IJ=I \cap J$.
	\end{thm}
	\begin{proof}
		Because $I$ and $J$ are both $d$-ideals, hence two sided, we get that $IJ \subseteq I \cap J$. To show the reverse inclusion, suppose that $x \in I \cap J$. Because $I$ and $J$ are coprime, there exists $a \in I$ and $b \in J$ such that $a+b=1$. Thus, $ax+bx=x$. Clearly, as $a \in I$ and $x \in J$, we see $ax \in IJ$. In addition, $bx=xb+d(x)d(b)$. Because $x \in I$ and $b \in J$, and $I$ and $J$ are both closed under $d$, we get that $xb+d(x)d(b) \in IJ$. Thus, $x=ax+bx=ax+xb+d(x)d(b) \in IJ$, so $I \cap J \subseteq IJ$.
	\end{proof}
	
	One interesting consequence of this theorem is that if $I$ and $J$ are coprime $d$-ideals, then $IJ=I \cap J = J \cap I = JI$. As it turns out, this is just a special case of the following stronger theorem:
	
	\begin{thm}
		If $I$ and $J$ are $d$-ideals, then $IJ$ is a $d$-ideal and $IJ=JI$.
	\end{thm}
	
	\begin{proof}
		Suppose $x \in IJ$. Then, by the definition of ideal multiplication:
		$$ x=\sum_i a_ib_i \text{ where }a_i \in I \text{ and }b_i \in J.$$
		First, we recall that $d$ is a derivation, so we get:
		$$ d(x) = \sum_i \left( a_i d(b_i) + d(a_i)b_i \right).$$
		Because $I$ and $J$ are both closed under $d$, we see that $d(x) \in IJ$, so $IJ$ is a $d$-ideal.
		
		To prove that $IJ=JI$, we recall that by $d$-commutativity,
		$$ x= \sum_i a_i b_i=\sum_i \left( b_i a_i + d(b_i)d(a_i) \right) \in JI.$$
		Thus $IJ \subseteq JI$, and the same argument shows $JI \subseteq IJ$, so $IJ=JI$.
	\end{proof}
	
	\subsection{Structure Theory of Artinian $d$-Algebras}
	
	Part of the structure theory for Noetherian $d$-algebras was initiated in [5], where the author has proven that the Hilbert basis theorem, a classical result in commutative algebras, for the setting of $d$-algebras. In this section, we initate the structure theory of Artinian $d$-algebras. In particular, we generalize the classical statement that any commutative Artinian ring may be decomposed as a direct product of local rings by proving that the corresponding statement holds for $d$-algebras. We give two separate proofs of this fact. To begin, let $A$ be an Artinian $d$-algebra.
	
	The first proof we give is nearly identical to the standard proof given in [1]. In particular, we define the Jacobson radical of $A$ as $J=\textnormal{Jac}(A)=M_1 \cap M_2 \cap \cdots \cap M_k$, where the $M_i$ are the distinct maximal ideals of $A$ (there is a finite number because $A$ is Artinian). Since the $M_i$ are clearly pairwise coprime, we can use Theorem \ref{coprime intersect} to say that $J=M_1\cdots M_k$. We require the following theorem from commutative algebra, which also holds in the general noncommutative case (and hence in the case of $d$-algebras).
	\begin{thm}
		For some $m \geq 1$, we have $J^m=0$.
	\end{thm}
	For a proof of the general noncommutative case, the reader is directed to [4]. We now recall that the Chinese remainder theorem holds in any ring with unity, regardless of commutativity. With this in mind, we can finally prove:
	\begin{thm}
		\label{local decomposition}
		$A$ may be written as a direct product of $k$ local $d$-algebras, where $k$ is the number of maximal ideals of $A$.
	\end{thm}
	
	\begin{proof}
		We begin by remarking that all maximal ideals are $d$-ideals, so their multiplication is commutative. Thus, $$0=J^m=\left( M_1 M_2 \cdots M_k \right)^m = M_1^m M_2^m \cdots M_k^m$$
		In addition, for any two distinct maximal ideals $M_i$ and $M_j$, we know that $M_i^m$ and $M_j^m$ are coprime. This can be shown by noting that if $M_i^m + M_j^m \neq A$, then $M_i^m+M_j^m$ is contained in some maximal ideal $M$. Then, suppose $x \in M_i$ and $y \in M_j$. Then $x^m \in M_i^m \subseteq M$, and maximal ideals are prime, so $x \in M$. Likewise, $y^m \in M_j^m \subseteq M$, so $y \in M$. Hence, $M_i \subseteq M$ and $M_j \subseteq M$, which gives $M_i = M_j$, a contradiction. Thus, by Theorem \ref{coprime intersect} above, we see that $0=M_1^m M_2^m \cdots M_k^m = M_1^m \cap M_2^m \cap \cdots \cap M_k^m$.
		
		Hence, by the Chinese remainder theorem:
		$$ A \cong A/0 = A/\left(M_1^m \cap M_2^m \cap \cdots \cap M_k^m \right) \cong A/M_1^m \times A/M_2^m \times \cdots A/M_k^m. $$
		
		For each $i$, the algebra $A/M_i^m$ is a local $d$-algebra with maximal ideal $M_i/M_i^m$.
	\end{proof}
	
	We now give a second proof of this theorem, which exposes some very useful information about $\Ker(d)$ along the way. In particular, set $K=\Ker(d)$, and suppose $K$ has $k$ distinct maximal ideals. Then, as $K$ is commutative and Artinian, we may decompose $K$ uniquely as $K_1\times\cdots\times K_k$, where each $K_i$ is local. Let $e_i$ be the idempotent element in $K$ corresponding to the identity of $K_i$. Thus, $K =e_1K + \cdots + e_kK$ and $e_1+\cdots+e_k=1$. Furthermore, we have $e_iK \cong K_i$ so each $e_iK$ is local.
	
	\begin{thm}
		There is a bijective correspondence between maximal ideals of $K$ and maximal ideals of $A$.
	\end{thm}
	\begin{proof}
		Let $M$ be a maximal ideal in $K$, and suppose $M'$ and $M''$ are maximal ideals of $A$ that contain $M$. Then, $M \subseteq M' \cap K \subset K$. Furthemore, $M' \cap K$ is clearly an ideal in $K$, and $M$ is maximal in $K$, so we must have $M = M' \cap K$. By the same logic, $M'' \cap K = M$. Now, suppose $x \in M'$. Then $d(x^2)=2xd(x)=0$, so $x^2 \in M' \cap K = M'' \cap K \subseteq M''$, and because maximal ideals are prime, we have $x \in M''$. Thus, $M' \subseteq M''$, and by symmetry, we have $M' = M''$. Hence, for each maximal ideal $M$ of $K$, there is a unique maximal ideal of $A$ that contains $M$. Furthermore, for each maximal ideal $M'$ of A, $M' \cap K$ is a maximal ideal of $K$ that is contained in $M'$, so the correspondence is bijective.
	\end{proof}
	\begin{cor}
		$A$ is local if and only if $K$ is local.
	\end{cor}
	
	We may now use this to give a simple proof of local decomposition:
	
	\begin{thm}
		$A$ may be decomposed as $A \cong e_1A \times \cdots \times e_kA$, where each $e_iA$ is local.
	\end{thm}
	\begin{proof}
		Clearly, because $e_1+\cdots+e_k=1$, we have $A = e_1A + \cdots + e_kA$. Then, because $e_ie_j=0$ for $i \neq j$, we get $e_iA \cap e_jA= \{0\}$ for $i \neq j$. Hence, we get that  $A \cong e_1A \times \cdots \times e_kA$.
		
		Now, let $d_i:e_iA \to e_iA$ denote the differential that $d$ induces on the $d$-algebra $e_iA$. Then, we have $\Ker(d_i) = e_iK$, so $\Ker(d_i)$ is local. Thus, by the previous lemma, $e_iA$ is local as well.
	\end{proof}

	\subsection{$d$-Algebras with Finite Defect}
	
	In this section, we deduce theorems on the structure of $d$-algebras based on a property that we shall call the defect of a $d$-algebra. We recall that $\Im(d)$ is a subalgebra of $\Ker(d)$, and moreover, for $v \in \Ker(d)$ and $d(a) \in \Im(d)$, we have $vd(a)=d(va) \in \Im(d)$, so $\Im(d)$ is in fact an ideal of $\Ker(d)$. Thus, we may form the quotient algebra $H(A):= \Ker(d) \big / \Im(d)$, which we shall call the \emph{Homology Algebra of A}.
	
	\newcommand{\deff}{\textnormal{def}}
	
	\begin{definition}
		Let $A$ be a $d$-algebra. If $H(A)$ is finite dimensional, then we define the \emph{defect} of $A$ as $\deff(A) := \dim  \left( H(A) \right)$, and we say that $A$ has finite defect. 
	\end{definition}
	
	In the case where $A$ is finite dimensional, $A$ clearly has finite defect, and furthermore we must have $\deff(A)=\dim(\Ker(d))-\dim(\Im(d))=\dim(A)-2\dim(\Im(d))$.
	
	\begin{thm}
		Suppose $A$ has finite defect, and $A \cong A_1 \times \cdots \times A_k$ where each $A_i$ is a $d$-algebra. Then $\deff(A)=\deff(A_1)+\cdots+\deff(A_k)$.
	\end{thm}
	\begin{proof}
		Let $d_i$ be the differential of $A_i$. Then, $\Ker(d) \cong \Ker(d_1)\times\cdots\times\Ker(d_k)$ and $\Im(d) \cong \Im(d_1)\times\cdots\times\Im(d_k)$. Thus, 
		$$ \frac{\Ker(d)}{\Im(d)} \cong \frac{\Ker(d_1)}{\Im(d_1)}\times\cdots\times\frac{\Ker(d_k)}{\Im(d_k)}.$$
		The result follows by counting dimensions.
	\end{proof}
	
	\begin{cor}
		Suppose $A$ is an Artinian $d$-algebra with finite defect, and suppose that $A$ has $k$ maximal ideals. Then $k \leq \deff(A)$.
	\end{cor}
	\begin{proof}
		By Theorem \ref{local decomposition}, we may write $A \cong A_1\times \cdots \times A_k$, where each $A_i$ is local. Furthermore, because $\Ker(d_i) \supset \Im(d_i)$, where the inclusion is strict, we have $\deff(A_i) \geq 1$. Thus, by the previous theorem, we get $\deff(A) = \deff(A_1) + \cdots +\deff(A_k) \geq k$.
	\end{proof}
	\begin{cor}
		Suppose $A$ is an Artinian $d$-algebra and $\deff(A)=1$. Then $A$ is local.
	\end{cor}
	
	
	\begin{thm}
		\label{def 1 structure}
		Suppose $A$ is finite dimensional, $\deff(A)=1$, and $\dim(\Im(d))=f$. Then, there exists a basis for $A$ of the form $\{1, v_1,\ldots, v_f,w_1,\ldots,w_f\}$, such that $v_i^2=w_i^4=0$. Furthemore, the set $\{v_1,\ldots,v_f,w_1,\ldots,w_f\}$ forms a basis for the unique maximal ideal of $M$.
	\end{thm}
	\begin{proof}
		Let $\{ v_1,\ldots,v_f\}$ be a basis for $\Im(d)$, and recall $\Im(d) \subseteq M$. Now, pick $\{z_1,\ldots,z_f\}$ such that $d(z_i)=v_i$.  Then, $z_1^2 \in \Ker(d)$, so we have:
		$$ z_1^2 = a_0 + a_1v_1 + \cdots + a_fv_f$$
		Thus, as $F$ is algebraically closed, we may pick some $\sqrt{a_0} \in F$ such that $(z_1+\sqrt{a_0})^4=0$, so we may replace $z_1$ with $w_1:=z_1+\sqrt{a_0}$ and note that $d(w_1)=v_1$. Likewise, we may replace each $z_i$ with a $w_i$ such that $w_i^4=0$ and $d(w_i)=v_i$. Since $\deff(A)=1$, we have $\dim(A)=2f+1$, so the set $\{1,v_1,\ldots,v_f,w_1,\ldots,w_f\}$ forms a basis of the given form. Since each $v_i$ and $w_i$ are nilpotent, they must all belong to the unique maximal ideal of $A$, so they must form a basis of the ideal.
	\end{proof}
	
	\subsection{The Case of Dimension 7}
	
	In this section, we classify all noncommutative $d$-algebras of dimension $7$. Thus, let $A$ be a noncommutative $d$-algebra such that $\dim(A)=7$. Thus, by Theorem \ref{dim im >= 2}, we must have $\dim(\Im(d)) \geq 3$, but $\dim(\Im(d)) < \dim(\Ker(d))$, so we must have $\dim(\Im(d))=3$ and $\dim(\Ker(d))=4$. Thus, $\deff(A)=1$, so $A$ is local. Let us suppose that $z_1$ and $z_2$ do not commute with each other, and let us set $v_1=d(z_1)$ and $v_2=d(z_2)$, so that $z_1z_2-z_2z_1=v_1v_2 \neq 0$. Then, as in the proof of Theorem \ref{dim im >= 2}, we see that $\{v_1,v_2,v_1v_2\}$ forms a basis for $\Im(d)$. As in the proof of \ref{def 1 structure}, we may choose $w_1$ and $w_2$ such that $w_1^4=w_2^4=0$ and $d(w_1)=v_1$ and $d(w_2)=v_2$. Finally, we note that $d(v_1w_2)=v_1v_2$, so the set $\{1,v_1,v_2,v_1v_2,w_1,w_2,v_1w_2\}$ forms a basis for $A$. For convinience, we set $v_3:=v_1v_2$ and $w_3:=v_1w_2$. Thus, $v_1v_3=v_2v_3=v_3^2=0$. In addition, $w_3$ is central, because $w_3x-xw_3=d(w_3)d(x)=v_3d(x)$, and we know that $v_3$ annahilates everything in $\Im(d)$.
	
	To aid our classification of $7$ dimensional $d$-algebras, we will define the following class of $d$-algebras as:
	$$ D(h,k,p) = P^2_0/(x_1^2-h\xi_1\xi_2, \ x_2^2-k\xi_1\xi_2, \ x_1x_2-p\xi_1\xi_2, \ \xi_1x_1, \ \xi_2x_2, \ \xi_1x_2-\xi_2x_1).$$
	Then the set $\{1, \ \xi_1, \ \xi_2, \ \xi_1\xi_2, \ x_1, \ x_2, \ \xi_1x_2\}$ forms a basis for $D(h,k,p)$, so it has dimension $7$. The objective of this section is to show that $A \cong D(0,0,0)$. To accomplish this, we will first show that $A \cong D(h,k,p)$ for some $h,k,p \in F$. We will then show that $D(h,k,p) \cong D(0,0,q)$ for some $q \in F$, and finally we will show that $D(0,0,q) \cong D(0,0,0)$.
	
	\begin{thm}
		There exist $h,k,p \in F$ such that $A \cong D(h,k,p)$.
	\end{thm}
	\begin{proof}
		Let $\{ 1, \ v_1, \ v_2, \ v_3, \ w_1, \ w_2, \ w_3\}$ be the basis described above, and let $M$ denote the unique maximal ideal of $A$ (so $v_i,w_i \in M$). We begin by noting that $v_3w_3=v_1^2v_2w_2=0$, $v_1w_3=v_1^2w_2=0$, and $w_3^2=v_1^2w_2^2=0$. Furthermore, $v_2w_3=v_1v_2w_2=v_3w_2$.
		
		Then, $d(v_1w_1)=v_1^2=0$, so $v_1w_1 \in \Ker(d) \cap M$. Thus, for some $a_1,a_2,a_3 \in F$, we have $v_1w_1=a_1v_1+a_2v_2+a_3v_3$. Mutliplying by $v_1$ gives $0=a_2v_3$, so $a_2=0$. Multiplying by $v_2$ gives us $v_3w_1=a_1v_3$. Repeating the process for $v_2w_2$ gives us $v_2w_2 = b_1v_1+b_2v_2+b_3v_3$ for some $b_1,b_2,b_3 \in F$. Multiplying by $v_2$ gives us $0=b_1v_3$, so $b_1=0$. Multiplying by $v_1$ gives us $v_3w_2=b_2v_3$.
		
		Next, we note that $d(w_1^2)=0$, so $w_1^2 \in \Ker(d) \cap M$. Thus, for some $h_1,h_2,h_2 \in F$, we have $w_1^2 = h_1v_1 + h_2v_2 + h_3 v_3$. Then, multiplying by $v_1$ gives us:
		\begin{align*}
		v_1w_1w_1 &= h_2v_3 \\
		(a_1v_1+a_3v_3)w_1 &= h_2v_3 \\			
		a_1(a_1v_1+a_3v_3) + a_3a_1v_3 &= h_2v_3.
		\end{align*}
		Thus, by comparing coefficients of $v_1$, we get $a_1^2=0$, so $a_1=0$, and by comparing coefficients of $v_3$, we get $0=h_2$. Thus, $v_3w_1=a_1v_3=0$, and $v_1w_1=a_3v_3$.
		
		Now we repeat this process with $w_2^2$. Since $d(w_2^2)=0$, we have $w_2^2=k_1v_1+k_2v_2+k_3v_3$ for some $k_1,k_2,k_3 \in F$. Multiplying by $v_2$ and expanding like we did with $w_1^2$ gives us $b_2=0$ and $k_1=0$. Thus, $v_3w_2=b_2v_3=0$ (so therefore $v_2w_3=v_3w_2=0$), and $v_2w_2=b_3v_3$.
		
		Hence, we have shown that multiplication by $v_3$ annihilates every basis element (except $1$). 
		
		Now, we note that $d(v_2w_1)=v_3$, so we have $v_2w_1=g_1v_1 + g_2v_2 + g_3v_3 + w_3$ for some $g_1,g_2,g_3 \in F$. Multiplying by $v_1$ gives $0=g_2v_3$, so $g_2=0$. Multiplying by $v_2$ gives $0=g_1v_3$, so $g_1=0$. Thus $v_2w_1=g_3v_3+w_3$.
		
		Now, we compute $d(w_1w_2)=v_1w_2+v_2w_1=w_3+(g_3v_3+w_3)=g_3v_3$. Thus, for some $p_1, p_2, p_3 \in F$, we have $w_1w_2=p_1v_1 + p_2v_2 + p_3v_3 + g_3w_3$. Multiplying by $v_1$ gives us:
		\begin{align*}
		(v_1w_1)w_2=p_2v_3 \ \ \  \implies \ \ \   (a_3v_3)w_2=p_2v_3 \ \ \  \implies \ \ \  0=p_2v_3 \ \ \ \implies \ \ \ p_2=0.
		\end{align*}
		Multiplying by $v_2$ gives us:
		\begin{align*}
		w_1(v_2w_2)=p_1v_3 \ \ \  \implies \ \ \   w_1(b_3v_3)=p_1v_3 \ \ \  \implies \ \ \  0=p_1v_3 \ \ \ \implies \ \ \ p_1=0.
		\end{align*}
		Thus $w_1w_2=p_3v_3+g_3w_3$. By $d$-commutativity, $w_2w_1=w_1w_2+v_3$, so by combining these two equations, we get $w_2w_1=(p_3+1)v_3+g_3w_3$. Multiplying by $v_2$ gives:
		\begin{align*}
		w_2w_1v_2=0 \ \ \  \implies \ \ \   w_2(g_3v_3+w_3)=0 \ \ \  \implies \ \ \  w_2w_3=0 .
		\end{align*}
		
		In addition, we see $w_1w_3=(w_1v_1)w_2=a_3v_3w_2=0$. Thus, $w_3$ annhilates every basis element except $1$. We now return to $w_1^2=h_1v_1+h_3v_3$. Upon multiplication (on the right) by $w_2$, we get:
		\begin{align*}
		w_1(w_1w_2)=h_1v_1w_2 \ \ \ \implies w_1(p_3v_3 + g_3w_3)=h_1w_3 \ \ \ \implies 0=h_1w_3 \implies h_1=0.
		\end{align*}
		Repeating the process with $w_2^2=k_2v_2+k_3v_3$ (except multiplying by $w_1$ instead of $w_2$):
		\begin{align*}
		(w_1w_2)w_2=k_2v_2w_1 \ \ \ &\implies (p_3v_3 + g_3w_3)w_2=k_2(g_3v_3+w_3) \\
		\ \ \ &\implies 0=k_2g_3v_3+k_2w_3 \implies k_2=0.
		\end{align*}
		
		At the moment, our multiplication table is given by Table \ref{mtable1}. 
		
		\begin{table}[h]
			\centering
			\caption{Multiplication Table}
			\label{mtable1}
			\begin{tabular}{|c|c|c|c|c|c|c|}
				\hline
				& $v_1$    & $v_2$        & $v_3$ & $w_1$               & $w_2$           & $w_3$ \\ \hline
				$v_1$ & $0$      & $v_3$        & $0$   & $a_3v_3$            & $w_3$           & $0$   \\ \hline
				$v_2$ & $v_3$    & $0$          & $0$   & $g_3v_3+w_3$        & $b_3v_3$        & $0$   \\ \hline
				$v_3$ & $0$      & $0$          & $0$   & $0$                 & $0$             & $0$   \\ \hline
				$w_1$ & $a_3v_3$ & $g_3v_3+w_3$ & $0$   & $h_3v_3$            & $p_3v_3+g_3w_3$ & $0$   \\ \hline
				$w_2$ & $w_3$    & $b_3v_3$     & $0$   & $(p_3+1)v_3+g_3w_3$ & $k_3v_3$        & $0$   \\ \hline
				$w_3$ & $0$      & $0$          & $0$   & $0$                 & $0$             & $0$   \\ \hline
			\end{tabular}
		\end{table}
		
		We now wish to get rid of $a_3$, $b_3$, and $g_3$, and in order to accomplish this, we are going to strategically pick a new basis. In particular, we set $z_1=w_1+g_3v_1+a_3v_2$, and $z_2=w_2+b_3v_1$. Then, we have $\{ 1, v_1, v_2, v_3, z_1, z_2, w_3\}$ form a basis for $A$, and it has the properties we desire. In particular, we have $d(z_i)=v_i$, and we may go through and reexamine the multiplication table as follows:
		\begin{itemize}
			\item $ v_1z_1 = v_1(w_1+g_3v_1+a_3v_2)=v_1w_1+a_3v_3=0$.
			\item $ v_2z_2 = v_2(w_2+b_3v_1)=v_2w_2+b_3v_3=0$.
			\item $ v_1z_2=v_1(w_2+b_3v_1)=v_1w_2=w_3$.
			\item $ v_2z_1 = v_2(w_1+g_3v_1+a_3v_2)=v_2w_1+g_3v_3=g_3v_3+w_3+g_3v_3=w_3$.
			\item $z_1z_2=(w_1+g_3v_1+a_3v_2)(w_2+b_3v_1)=w_1w_2+g_3v_1w_2+a_3b_3v_3=(p_3+a_3b_3)v_3$.
			\item $z_1^2=(w_1+g_3v_1+a_3v_2)^2=w_1^2=h_3v_3$.
			\item $z_2^2=(w_2+b_3v_1)^2=w_1^2=k_3v_3$.
		\end{itemize}
		
		Thus, letting $h=h_3$, $k=k_3$, and $p=p_3+a_3b_3$, this set of relations is exactly the set that we want. In particular, we will show that $A \cong D(h,k,p)$.
		
		To see this, first set $P=P^2_0$. Then, define the $d$-algebra homomorphosm $\phi: P \to A$ by $\phi(x_1)=z_1$, $\phi(x_2)=z_2$, $\phi(\xi_1)=v_1$, and $\phi(\xi_2)=v_2$. This map is well defined by the $d$-commutativity of $A$ and is clearly surjective. Thus, it suffices to prove that the ideal $I=(x_1^2-h\xi_1\xi_2, \ x_2^2-k\xi_1\xi_2, \ x_1x_2-p\xi_1\xi_2, \ \xi_1x_1, \ \xi_2x_2, \ \xi_1x_2-\xi_2x_1)$ is equal to $\Ker(\phi)$. The calculations listed above readily verify that every generator of $I$ lies in $\Ker(\phi)$, so $I \subseteq \Ker(\phi)$. But it is easy to see that $P^2_0/I$ is of dimension $7$, so $\Ker(\phi)$ cannot strictly include $I$, for then $P^2_0/\Ker(\phi)$ would have dimension smaller than $7$, which contradicts $P^2_0/\Ker(\phi) \cong A$. Thus, $I= \Ker(\phi)$, so $A \cong P_0^2/I=D(h,k,p)$.
	\end{proof}
	\begin{thm}
		For all $h,k,p \in F$, there exists $q \in F$ such that $D(h,k,p) \cong D(0,0,q)$. 
	\end{thm}
	\begin{proof}
		If $h=k=0$, then $q=p$ and we are done, so we consider the case of $k \neq 0$ (The case of $h \neq 0$ is completely analogous). Then, consider the polynomial $f(x)=kx^2+x+h$, and let $\alpha$ and $\beta$ be the roots of $f(x)$ in $F$ (which we can extract because $F$ is algebraically closed). From this, we know that $\alpha+\beta=k^{-1}$, and $\alpha\beta=hk^{-1}$, so $k(\alpha + \beta)=1$ and $k\alpha\beta=h$. 
		
		For convinience, set $u_1=x_1+\alpha x_2$ and $u_2=x_1+\beta x_2$. Now, construct the $d$-algebra homomorphism $\phi:P_0^2 \to D(h,k,p)$ by setting $\phi(x_1)=\sqrt{p}d(u_1)+u_1$ and $\phi(x_2)=\sqrt{p}d(u_2)+u_2$. Then, $\phi(\xi_1)=d(u_1)=\xi_1 + \alpha \xi_2$ and $\phi(\xi_2)=d(u_2)=\xi_1 + \beta \xi_2$. We note that $\alpha + \beta = k^{-1} \neq 0$, so $\alpha \neq \beta$, so $\phi(x_1)$ and $\phi(x_2)$ are linearly independent, and consequently, it is easy to see that $\phi$ must be surjective.
		
		Now, we set $q=\alpha k$. We would like to show that $\Ker(\phi)$ is equal to $I=(x_1^2 \ x_2^2, \ x_1x_2-q\xi_1\xi_2, \ \xi_1x_1, \ \xi_2x_2, \ \xi_1x_2-\xi_2x_1)$, because then we would have $D(h,k,p) \cong P_0^2/I=D(0,0,q)$.
		
		To do this, we list out the calculations necessary:
		$$\phi(\xi_1x_1)=d(u_1)(\sqrt{p}d(u_1) + u_1) = u_1d(u_1)=(x_1+\alpha x_2)(\xi_1 + \alpha \xi_2)=\alpha x_1\xi_2 + \alpha x_2 \xi_1 = 0$$
		$$\phi(\xi_2x_2)=d(u_2)(\sqrt{p}d(u_2) + u_2) = u_2d(u_2)=(x_1+\beta x_2)(\xi_1 + \beta \xi_2)=\beta x_1\xi_2 + \beta x_2 \xi_1 = 0.$$
		\begin{align*}
		\phi(\xi_1x_2-\xi_2x_1) &= d(u_1)(\sqrt{p}d(u_2)+u_2) + d(u_2)(\sqrt{p}d(u_1)+u_1) = u_2d(u_1)+u_1d(u_2) \\
		&= d(u_1u_2)=d(x_1^2+\beta x_1x_2 + \alpha x_2x_1 + x_2^2)=d(\beta p \xi_1\xi_2 + \alpha(p+1)\xi_1\xi_2)=0.
		\end{align*}
		$$ \phi(x_1^2)=(\sqrt{p}d(u_1)+u_1)^2=u_1^2=(x_1+\alpha x_2)^2 = x_1^2+\alpha \xi_1\xi_2 + \alpha^2 x_2^2 = (h + \alpha + k\alpha^2)\xi_1\xi_2=0.$$
		$$ \phi(x_2^2)=(\sqrt{p}d(u_2)+u_2)^2=u_2^2=(x_1+\beta x_2)^2 = x_1^2+\beta \xi_1\xi_2 + \beta^2 x_2^2 = (h + \beta + k\beta^2)\xi_1\xi_2=0.$$
		\begin{align*}
		\phi(x_1x_2-\alpha k \xi_1\xi_2) &= (\sqrt{p}d(u_1)+u_1)(\sqrt{p}d(u_2)+u_2) - \alpha k d(u_1)d(u_2) \\
		&= pd(u_1)d(u_2) + \sqrt{p}(u_1d(u_2)+u_2d(u_1)) + u_1u_2 - \alpha k d(u_1)d(u_2))  \\
		&= pd(u_1)d(u_2) + u_1u_2 -\alpha kd(u_1)d(u_2) \\
		&= p (\alpha + \beta)\xi_1\xi_2 + (x_1+\alpha x_2)(x_1+ \beta x_2) - \alpha k (\alpha + \beta)\xi_1 \xi_2 \\ 
		&= p (\alpha + \beta)\xi_1\xi_2 + h\xi_1\xi_2 + \beta p\xi_1\xi_2 + \alpha (p+1)\xi_1\xi_2 +k\alpha\beta \xi_1\xi_2 - \alpha k (\alpha + \beta)\xi_1 \xi_2 \\ 
		&= h \xi_1\xi_2 + \alpha \xi_1\xi_2 + k\alpha\beta \xi_1\xi_2 - \alpha k (\alpha + \beta)\xi_1 \xi_2 \\
		&= h\xi_1\xi_2 + \alpha \xi_1\xi_2 + h \xi_1\xi_2 - \alpha \xi_1 \xi_2 = 0.
		\end{align*}
		
		Thus, all the generators of $I$ lay in $\Ker(\phi)$, so $I \subseteq \Ker(\phi)$, and it is easily checked that $\dim(P_0^2/I)=7$, so we must have $I=\Ker(\phi)$, and thus $D(0,0,q) \cong P_0^2/\Ker(\phi) = D(h,k,p)$.	
	\end{proof}
	\begin{thm}
		For all $q \in F$, $D(0,0,q) \cong D(0,0,0)$.
	\end{thm}
	\begin{proof}
		We use the same strategy as before. Consider the map $\phi: P_0^2 \to D(0,0,q)$ defined by $\phi(x_1)=\sqrt{q}\xi_1 + x_1$ and $\phi(x_2) = \sqrt{q}\xi_2 + x_2$. Then $\phi(\xi_1)=\xi_1$ and $\phi(\xi_2)=\xi_2$. Clearly $\phi$ is surjective, so we wish to show that $\Ker(\phi)$ equals $I=(x_1^2, \ x_2^2, \ x_1x_2, \ \xi_1x_1, \ \xi_2x_2, \ \xi_1x_2-\xi_2x_1 )$. The necessary calculations are:
		\begin{itemize}
			\item $\phi(x_1^2)= (\sqrt{q}\xi_1+x_1)^2=x_1^2=0$.
			\item $\phi(x_2^2)= (\sqrt{q}\xi_2+x_2)^2=x_2^2=0$.
			\item $\phi(x_1x_2)= (\sqrt{q}\xi_1 + x_1)(\sqrt{q}\xi_2 + x_2) = q\xi_1\xi_2 + \sqrt{q}(\xi_1x_2 + \xi_2x_1) + x_1x_2 = 0$.
			\item $\phi(\xi_1x_1)= \xi_1(\sqrt{q}\xi_1 + x_1)=0$.
			\item $\phi(\xi_2x_2)= \xi_2(\sqrt{q}\xi_2 + x_2)=0$.
			\item $\phi(\xi_1x_2)=\xi_1(\sqrt{q}\xi_2 + x_2) = \sqrt{q}\xi_1\xi_2 + \xi_1x_2$.
			\item $\phi(\xi_2x_1)=\xi_2(\sqrt{q}\xi_1 + x_1) = \sqrt{q}\xi_1\xi_2 + \xi_2x_1 = \sqrt{q}\xi_1\xi_2 + \xi_1x_2 = \phi(\xi_1x_2)$.
		\end{itemize}
		
		Thus, by the same argument as before, we have $\Ker(\phi)=I$, and so $D(0,0,q) \cong P_0^2/I = D(0,0,0)$.
	\end{proof}
	
	Hence, we have shown that up to isomorphism, the only noncommutative dimension $7$ $d$-algebra is $D(0,0,0)$.
	
	\section{Lie Algebras in $\sVec_2$}
	
	\subsection{Definition}
	In this section, we give and motivate the definition of a Lie algebra in $\sVec_2$. In particular, have the following:
	
	\begin{definition} \label{lie alg def}
		A Lie algebra in $\sVec_2$ is an object $L \in \sVec_2$ together with a bilinear bracket operation $[,]:L \otimes L \to L$ which satisfies the following conditions:
		\begin{enumerate}
			\item $d$ is a derivation over $[,]$:
			$$ d[x,y] = [dx,y] + [x,dy].$$
			\item Antisymmetry:
			$$ [x,y] + [y,x] + [dy,dx] = 0 \text{ for all }x,y \in L.$$
			\item Jacobi Identity:
			$$ [x,[y,z]] + [y,[x,z]] + [dy,[dx,z]] = [[x,y],z] \text{ for all }x,y,z \in L.$$
			\item PBW requirement:
			$$ \text{if }dx=0\text{, then }[x,x]=0.$$
		\end{enumerate}
	\end{definition}
	
	To motivate this defintion, we recall that in a symmetric tensor category $\mathcal{C}$ over a field of characteristic $0$, a Lie algebra, as defined in [3], is an object $L \in \mathcal{C}$ together with a bracket morphism $\beta : L \otimes L \to L$ which satisfies anti-commutativity:
	$$ \beta \circ (\id + c) = 0$$
	and the Jacobi identity:
	$$ \beta \circ (\beta \otimes \id) \circ (\id + \sigma + \sigma^2) = 0$$
	where $c=c_{L,L}:L \otimes L \to L \otimes L$ is the commutativity map, and $\sigma:L \otimes L \otimes L \to L \otimes L \otimes L$ is the permutation $(123)$; more concretely $\sigma = (c \otimes \id ) \circ (\id \otimes c)$
	
	This definition suffices for when the base field has characteristic zero because in such cases, the Poincare-Birkhoff-Witt theorem holds (see [3]). However, the PBW theorem is known to fail for a variety of cases when the characteristic of the base field is positive. In [2], Pavel Etingof demonstrates that for any prime $p$, there exists a symmetric tensor category over a field of charactersitic $p$ for which the PBW theorem fails to hold.
	
	For example, in characteristic $2$, the PBW theorem fails to hold in the standard category $\Vec$. To ensure that PBW holds, the additional condition of $[x,x]=0$ for all $x$ is added to the definition of a Lie algebra in $\Vec$.
	
	Likewise, in characteristic $3$, the PBW theorem fails to hold in the category $\sVec$. For this case, the addition condition of $[[x,x],x]=0$ for all odd $x$ must be added to the definition of a Lie algebra in $\sVec$ in order to ensure that PBW holds.
	
	The definition of a Lie algebra in $\sVec_2$ follows along the same lines. It is defined through the axioms given above, which are taken from [3], and an additional axiom is imposed to ensure that PBW holds. 
	
	We now motivate the definition of a Lie algebra in $\sVec_2$. In particular, suppose $L \in \sVec_2$ and $\beta : L \otimes L \to L$ is the bracket operation. We write $[x,y]$ for $\beta(x \otimes y)$. Since $\beta$ is a morphism in $\sVec_2$, it must commute with the operator $d$, and since $d_{L \otimes L} = 1 \otimes d + d \otimes 1$, we must have:
	$$ d[x,y] = [dx,y] + [x, dy] \text{ for all }x,y \in L$$
	so $d$ is a derivation over $[,]$.
	
	Then, since $c(x \otimes y) = y \otimes x + dy \otimes dx$ the antisymmetry axiom becomes:
	$$ [x,y] + [y,x] + [dy, dx] = 0 \text{ for all }x,y \in L.$$
	
	To expand out the Jacobi identity, we first see:
	\begin{align*}
		\sigma(x \otimes y \otimes z) &= (c \otimes \id ) \circ (\id \otimes c) (x \otimes y \otimes z) \\
		&= (c \otimes \id) (x \otimes z \otimes y + x \otimes dz \otimes dy) \\
		&= z \otimes x \otimes y + dz \otimes dx \otimes y + dz \otimes x \otimes dy \\
		\sigma^2(x \otimes y \otimes z) &= (\id \otimes c ) \circ (c \otimes \id) (x \otimes y \otimes z) \\
		&= (\id \otimes c) (y \otimes x \otimes z + dy \otimes dx \otimes z)) \\
		&= y \otimes z \otimes x + y \otimes dz \otimes dx + dy \otimes z \otimes dx.
	\end{align*}
	
	Then, we see that the Jacobi identity for $\sVec_2$ becomes:
	
	\begin{align*}
		[[x,y],z] &+ [[z,x],y] + [[dz,dx],y] + [[dz,x],dy]  \\ 
		&+ [[y,z],x] + [[y,dz],dx] + [[dy,z],dx] = 0 \text{ for all }x,y,z \in L.
	\end{align*}
	
	Through repeated application of the fact that $d$ is a derivation over $[,]$ and the antisymmetry rule, this identity can be shown to be equivalent to the Jacobi identity stated in definition \ref{lie alg def}.
	
	If we define $\ad_{x}: L \to L$ for $x \in L$ by $\ad_{x}(y)=[x,y]$, then we may write the Jacobi identity as:
	$$ \ad_x \circ \ad_y - \ad_y \circ \ad_x - \ad_{dy} \circ \ad_{dx} = \ad_{[x,y]}.$$
	
	The reader should note the similarity between this formulation of the Jacobi identity for $\sVec_2$ and the standard Jacobi identity for $\Vec$:
	$$ \ad_x \circ \ad_y - \ad_y \circ \ad_x = \ad_{[x,y]}.$$
	
	We will see in the next section, after the statement of PBW has been properly formulated, why the fourth condition in definition $\ref{lie alg def}$ is necessary for PBW to hold, and hence why it is included in the definition of a Lie algebra in $\sVec_2$.
	
	\subsection{Tensor, Symmetric, and Universal Enveloping Algebras in $\sVec_2$}
	
	To properly formulate the PBW theorem in the category $\sVec_2$, we first need the notions of a tensor, symmetric, and universal algebra in $\sVec_2$. Since $\sVec_2$ is a symmetric tensor category, each of these notions has a natural definition, as in [3]. 
	
	\subsubsection{Tensor Algebras}
	
	In particular, suppose $L \in \sVec_2$. Then, as usual, we define the tensor algebra $T(L)$ as:
	$$ T_n(L) = \underbrace{L\otimes\cdots\otimes L}_{n \textnormal{ times}}$$
	
	$$ T(L)=\bigoplus_{n=0}^\infty T_n(L)$$
	
	where $T_0(L) = F$.
	
	Since the operator $d_L:L \to L$ may be uniquely extended to an operator $d_T: T(L) \to T(L)$, the tensor algebra $T(L)$ is then naturally a object of $\sVec_2$. Furthermore, it forms an associative algebra in $\sVec_2$ (where the multiplication $m:T(L) \otimes T(L) \to T(L)$ is just the tensor product), but not necessarily a commutative algebra in $\sVec_2$ (i.e. a $d$-algebra). 
	
	\subsubsection{Symmetric Algebras}
	
	Then, as in any symmetric tensor category, we define the symmetric algebra $S(L)$ as the quotient of the tensor algebra by the ideal generated by the image of the morphism $\id-c$ where $c=c_{L,L}:L \otimes L \to L \otimes L$ is the commutativity map. Specifically, define the ideal $I(L)$ as:
	$$ I(L) = \langle (\id-c)(x \otimes y) \mid x,y \in L \rangle = \langle x \otimes y + y \otimes x + dy \otimes dx \mid x,y \in L \rangle $$
	where we replaced minus signs with plus signs because $\char(F)=2$. Then, we may define the symmetric algebra $S(L)$ as:
	$$ S(L) = T(L)/I(L).$$
	
	Then $S(L)$ is naturally a commutative algebra ($d$-algebra) in $\sVec_2$. To understand the structure of $S(L)$, we examine the case where $L$ is finite dimensional. The case when $L$ is infinite dimensional is completely analogous. 
	
	
	In particular, suppose $\dim(L)=m$ and $\dim(\Im(d))=k$. Then, let $\{ v_1,\ldots, v_k \}$ be an ordered basis for $\Im(d)$ and append elements $ v_{k+1},\ldots,v_m $ so that $\{ v_1,\ldots,v_m \}$ is an ordered basis for $L$. Then, a basis for $S(L)$ is all monomials of the form :
	$$ v_1^{e_1}\otimes \cdots\otimes v_k^{e_k}\otimes  v_{k+1}^{e_{k+1}}\otimes \cdots\otimes v_m^{e_m} $$
	
	such that $e_1,\ldots,e_k \in \{0,1\}, \ e_{k+1},\ldots,e_{m} \in \mathbb{N}$. 
	
	An informal explanation which can easily be made rigorous as to why this set forms a basis is that because $S(L)$ is a $d$-algebra, the elements $v_1,\ldots,v_k$ are all central in $S(L)$ and satisfy $v_i^2 = 0$. Thus, given a monomial $v_{i_1} \otimes \ldots \otimes v_{i_r}$, the terms of the form $v_i$ for $1 \leq i \leq k$ may be ``pulled" to the front of the product, and the other terms may be rearranged using the rule $v_i \otimes v_j = v_j \otimes v_i + dv_j \otimes dv_i$. Once the $v_1,\ldots,v_k$ are all pulled to the front, any $v_i, 1 \leq i \leq k$ with exponent higher that $1$ becomes $0$, so in order to form a basis, the exponents of the $v_i, 1 \leq i \leq k$ must be restricted to $0$ and $1$.
	
	
	\subsubsection{Universal Enveloping Algebras}
	
	Finally, suppose that $L$ is a Lie algebra in $\sVec_2$ with bracket operation $\beta : L \otimes L \to L$, where we write $\beta(x \otimes y) = [x,y]$. Then, as in any symmetric tensor category, we define the universal enveloping algebra $U(L)$ as the quotient of the tensor algebra by the ideal generated by the image of the morphism $\id-c-\beta$. Specifically, we define the ideal $J(L)$ as:
	$$ J(L) = \langle (\id - c - \beta )(x \otimes y) \mid x,y \in L \rangle  = \langle x \otimes y + y \otimes x + dy \otimes dx + [x,y] \mid x,y \in L \rangle.$$
	
	Then we may define the univeral enveloping algebra $U(L)$ as:
	$$ U(L) = T(L)/J(L).$$
	
	The univeral enveloping algebra is natually an associative algebra in $\sVec_2$, but need not be a $d$-algebra.
	
	$U(L)$ also inherits a natural filtration from $T(L)$, which we will denote $U_n(L)$, $n \geq 0$, where $U_n(L)$ is spanned by monomials of degree $\leq n$. Then, the associated graded algebra of $U(L)$ is as usual:
	$$ \gr U(L) =  F \oplus \left( \bigoplus_{n \geq 1} U_n(L)/U_{n-1}(L) \right).$$
	
	Since $x \otimes y + y \otimes x + dy \otimes dx = [x,y] \in U_1(L)$ for all $x,y \in L$, it follows that $x \otimes y + y \otimes x + dy \otimes dx = 0 $ in $U_2(L)/U_1(L)$, and hence the same equality holds in $\gr U(L)$. Thus, $\gr U(L)$ is a commutative algebra in $\sVec_2$ (a $d$-algebra). 
	
	There is a natural inclusion $i:L \to \gr U(L)$, and because $\gr U(L)$ is $d$-commutative, it follows that $i$ can be uniquely extended to a $d$-algebra homomorphism $\tilde{i}: S(L) \to \gr U(L)$. Then, as in any symmetic tensor category, the PBW theorem takes the form:
	
	\begin{thm}[Categorical PBW]
		The natural map $\tilde{i}:S(L) \to \gr U(L)$ is an isomorphism.
	\end{thm}
	
	Since $\gr U(L)$ and $U(L)$ are isomorphic as vector spaces, an equivalent statement is that a basis of $S(L)$ is lifted to a basis of $\gr U(L)$, which is in turn a basis for $U(L)$. 
	
	Thus, suppose that $L$ is finite dimensional (once again, the infinite dimensional case is analogous). Then, let $\{ v_1,\ldots, v_k \}$ be an ordered basis for $\Im(d)$, and append elements $\{ v_{k+1}, \ldots, v_m \}$ such that $\{ v_1,\ldots, v_m \}$ is an ordered basis for $L$. 
	
	Then, we say that a monomial in $T_n(L)$ is a \textbf{standard monomial} if it is of the form:
	$$ v_1^{e_1}\otimes \cdots\otimes v_k^{e_k}\otimes  v_{k+1}^{e_{k+1}}\otimes \cdots\otimes v_m^{e_m} $$
	
	such that $e_1,\ldots,e_k \in \{0,1\}, \ e_{k+1},\ldots,e_{m} \in \mathbb{N}$. The reader should recognize this as the basis given above for $S(L)$. Then, we formula an equivalent statement to PBW for $\sVec_2$ as:
	
	\begin{thm}[PBW for $\sVec_2$]
		The set of standard monomials forms a basis for $U(L)$. Here the term ``standard monomial" is taken to mean the image of a standard momomial under the projection from $T(L)$ to $U(L)$.
	\end{thm}
	
	This has a very important corollary:
	\begin{cor}
		The natural map from $L$ to $U(L)$ is an injection.
	\end{cor}
	
	Now that we have properly formulated PBW for $\sVec_2$, we can explain why the fourth condition in definition \ref{lie alg def} is necessary. In particular, suppose that $L$ satisfies the first three conditions of definition \ref{lie alg def} and that the PBW theorem holds. Then, the natural map from $L$ to $U(L)$ in an injection. 
	
	Now, suppose $x \in L$ and $dx=0$. Then, we have $[x,x]=x \otimes x + x \otimes x + dx \otimes dx + [x,x] \in J(L)$. Thus, $[x,x] \in J(L)$, so $[x,x]=0$ in $U(L)$. Since $[x,x] \in L$ and the natural map from $L$ to $U(L)$ is an injection, it follows that $[x,x]=0$.
	
	Therefore, if $L$ satisfies the first three conditions of definition \ref{lie alg def}, then $L$ must also satisfy the fourth condition in order for PBW to hold. As we will prove in the next section, these four conditions suffice to ensure that PBW holds. Thus, defintion \ref{lie alg def} is the ``right" choice for the definition of a Lie algebra in $\sVec_2$.
	
	\subsection{Proof of PBW}
	
	We now suppose that $L$ is a finite dimensional Lie algebra in $\sVec_2$. As before, $\{ v_1,\ldots,v_m \}$ is an ordered basis for $L$, where $\{ v_1,\ldots, v_k \}$ is an ordered basis for $\Im(d)$. To ease notation, we write $T_n=T_n(L)$, $T=T(L)$, $J=J(L)$, and $U=U(L)$.
	
	We now prove PBW for $\sVec_2$:
	 \begin{thm}[PBW]
	 	The set of (images of) standard monomials form a basis for universal enveloping algebra $U$.
	 \end{thm}

	 \subsubsection{Span}
	 
	 We begin by proving that such monomials span $U$. To do so, we first suppose that we have a monomial $\alpha = v_{i_1}\otimes \cdots \otimes v_{i_r} \in T_r$. (The images of) such monomials clearly span $U$ as they span $T$. We define the defect of such a monomial to be the number of indices that are out of order; that is, the number of pairs $(j,j')$ such that $j>j'$ but $i_j < i_{j'}$. In addition, we define that $K-$degree of such a monomial to be the number of elements $v_{i_j}$ in the monomial such that $d(v_{i_j}) \neq 0$. We will omit tensor signs whenever convinient to save space. 
	 
	 To prove that standard monomials span, we first prove the following useful lemma:
	 
	 \begin{lem}
	 	Every monomial $\alpha$ in $T$ is equivalent modulo $J$ to a sum of monomials with ordered indices.
	 \end{lem}
	 
	 If the indices of $\alpha$ are not ordered, then there must exist some index $i_j$ such that $i_{j+1} < i_j$. Then, we have:
	 
	 $$ v_{i_{j}}  v_{i_{j+1}} + v_{i_{j+1}}  v_{i_{j}}  + dv_{i_{j+1}}  dv_{i_j} + [v_{i_j}, v_{i_{j+1}}] \in J.$$
	 
	 Thus,
	 \begin{align*}
	 v_{i_1} \cdots  v_{i_r} &= \big( v_{i_1}\cdots v_{i_{j-1}}  \cdot (v_{i_{j}}  v_{i_{j+1}} + v_{i_{j+1}}  v_{i_{j}} + dv_{i_{j+1}}  dv_{i_j} + [v_{i_j}, v_{i_{j+1}}])\cdot v_{i_{j+2}}\cdots v_{i_r} \big)  \\ &+ \big( v_{i_1}\cdots v_{i_{j+1}}\cdot v_{i_{j}}\cdots v_{i_r} \big) + \big(v_{i_1}\cdots v_{i_{j-1}}dv_{i_{j+1}}dv_{i_{j}} v_{i_{j+2}}\cdots v_{i_r} \big) \\ &+  \big(v_{i_1}\cdots v_{i_{j-1}}[v_{i_j},v_{i_{j+1}}] v_{i_{j+2}}\cdots v_{i_r} \big).
	 \end{align*}
	 
	 The first summand lies in $J$, the second summand has smaller defect, the third summand has lower $K-$degree, and the fourth summand has lower degree (as a tensor monomial). Hence, do induction on the degree of the tensor, then for each fixed degree do induction on the $K-$degree, and for each fixed $K-$degree do induction on defect.  Thus, modulo $J$, the monomial $\alpha$ is equivalent to the sum of the final three summands, which by induction must be equivalent to a sum of standard monomials. (Note: this ignores the base cases, which are all trivial to check)
	 
	 Hence we have shown that (the images of) monomials with ordered indices, i.e. those of the form $v_1^{e_1}\cdots v_m^{e_m}$, span $U$. To show that standard monmials span, we must show that such monomials with $e_i \in \{0,1\}$ for $1 \leq i \leq k$ must span. To do this, we pick $w_i \in L$ for $1 \leq i \leq k$ such that $d(w_i)=v_i$ (which we may do as the $v_i$ form a basis for $\Im(d)$). Then, we note that $w_i \otimes w_i + w_i \otimes w_i + dw_i \otimes dw_i + [w_i,w_i] \in J$, so $v_i \otimes v_i + [w_i,w_i] \in J$. Thus, we have $v_i^2 \equiv [w_i,w_i] \pmod J$  for $1 \leq i \leq k$. Thus, if $1 \leq i \leq k$ and $e_i \geq 2$, then the term $v_i^{e_i}$ may be reduced modulo $J$ to a tensor of lower degree. Then we do induction on the degree of the tensor to show that standard monomials must span.

	 \subsubsection{Linear Independence}
	 
	 We now prove that standard monomials are linearly independent in $U$. This is the tougher assertion. To do this, we first suppose that there exists some linear map $P:T \to T$ such that $P$ acts as the identity on standard monomials, and furthermore, whenever $i_{j}\geq i_{j+1}$, we have:
	 \begin{align*}
	 P(v_1\cdots v_{i_j}v_{i_{j+1}}\cdots v_{i_n}) &= P(v_1\cdots v_{i_{j+1}}v_{i_{j}}\cdots v_{i_n}) \\
	 &+ P(v_1\cdots dv_{i_{j+1}}dv_{i_{j}}\cdots v_{i_n}) \\
	 &+ P(v_1\cdots [v_{i_{j}},v_{i_{j+1}}]\cdots v_{i_n}).
	 \end{align*}
	 
	 If we can show that such a $P$ exists, then we must have $P(J)=0$, whereas $P$ acts as the identity on linear combinations of standard monomials. Thus, if we let $S$ denote the span of standard monomials, then the existence of such a $P$ would imply that $S \cap J = \{ 0 \}$, so no nontrivial linear combination of standard monomials is zero in $U$, hence such monomials must be linearly independent. 
	 
	 We note that if such a $P$ existed, then it would satisfy 
	 
	 \begin{align*}
	 P(y_1\cdots y_{i_j}y_{i_{j+1}}\cdots y_{i_n}) &= P(y_1\cdots y_{i_{j+1}}y_{i_{j}}\cdots y_{i_n}) \\
	 &+ P(y_1\cdots dy_{i_{j+1}}dy_{i_{j}}\cdots y_{i_n}) \\
	 &+ P(y_1\cdots [y_{i_{j}},y_{i_{j+1}}]\cdots y_{i_n})
	 \end{align*}
	 
	 for \emph{any} $y_{i_j}$ in $L$. This is because we may expand each $y_{i_j}$ out in terms of the basis $\{ v_i : 1 \leq i \leq m \}$, at which point the linearity of $P$ expresses such an identity in terms of basis elements.
	 
	 Hence the problem reduces to showing that such a $P$ exists, which then reduces to defining $P$ on monomials. We do induction on the degree $n$ of the monomial. We define $P$ as the identity on $T_0$ and $T_1$, in which case the first condition on $P$ holds since every monomial in $T_0$ and $T_1$ are standard, and the second condition holds vacuously.
	 
	 Next, we suppose suppose $n \geq 2$, and we assume for induction that $P$ is well defined on $\bigotimes_{r=0}^{n-1}T_r \subset T$. Pick a monomial $\alpha =  v_{i_1}\otimes \cdots \otimes v_{i_n} \in T_n$.
	 
	 We remark that if the $K$-degree of $\alpha$ is zero, then $\alpha$ is an element of $T(\Ker(d))$, the tensor algebra of $\Ker(d)$. Since $\Ker(d)$ is a standard lie algebra with the property that $[x,x]=0$ for $x \in \Ker(d)$, the usual proof of PBW shows that there must exist a well-defined map $P_K : T(\Ker(d)) \to T(\Ker(d))$ satisyfying the desired properties. Thus, for $\alpha \in T(\Ker(d))$, we \emph{define} $P(\alpha)$ by $P(\alpha)=P_K(\alpha)$. Hence, $P$ is well defined for the case where the $K$-degree of $\alpha$ is zero.
	 
	 Furthemore, when the defect of $\alpha$ is zero, $\alpha$ is of the form $v_1^{e_1}\cdots v_m^{e_m}$. If $\alpha$ is standard then we define $P(\alpha)=\alpha$. If $\alpha$ is not standard, then we must have $e_i \geq 2$ for some $0 \leq i \leq k$. Then, we define  $P(\alpha)$ by:
	 
	 \begin{align*}
	 P(\alpha) = P(v_1^{e_1} \cdots v_i^2 v_i^{e_i-2} \cdots v_m^{e_m}) := P(v_1^{e_1} \cdots [w_i,w_i] v_i^{e_i-2} \cdots v_m^{e_m})
	 \end{align*}
	 
	 where $w_i$ is chosen such that $d(w_i)=v_i$.
	 
	 The term $v_1^{e_1} \cdots [w_i,w_i] v_i^{e_i-2} \cdots v_m^{e_m}$ has lower degree, so by induction, $P$ has already been defined on it. A small amount of extra work must be put in to ensure that this definition of $P$ for monomials of defect zero is well defined. That is, in the event that we have two distinct $1 \leq i,j \leq k$ (say, $i < j$) such that $e_i,e_{j} \geq 2$, we must show that:
	 $$ P(v_1^{e_1} \cdots [w_i,w_i] v_i^{e_i-2} \cdots v_m^{e_m}) = P(v_1^{e_1} \cdots [w_{j},w_j] v_j^{e_j-2} \cdots v_m^{e_m})$$
	 
	 To do this, we simply note that since $P$ is already defined on tensors of degree less than that of $\alpha$, and furthermore on such tensors it satisfies the properties it is meant to satisfy, we thus have:
	 
	 \begin{align*}
	 P(v_1^{e_1} \cdots [w_i,w_i] v_i^{e_i-2} \cdots v_m^{e_m}) & & & &
	 \end{align*}
	 \begin{align*}
	 &= P(v_1^{e_1} \cdots [w_i,w_i] v_i^{e_i-2} \cdots v_j^{e_j} \cdots v_m^{e_m}) \\
	 &= P(v_1^{e_1} \cdots [w_i,w_i] v_i^{e_i-2} \cdots (v_jv_j)v_j^{e_j-2} \cdots v_m^{e_m})  \\
	 & \ \ \ \ + 	P(v_1^{e_1} \cdots [w_i,w_i] v_i^{e_i-2} \cdots (w_jw_j + w_jw_j + v_jv_j + [w_j,w_j])v_j^{e_j-2} \cdots v_m^{e_m}) \\
	 &= P(v_1^{e_1} \cdots [w_i,w_i] v_i^{e_i-2} \cdots [w_j,w_j]v_j^{e_j-2} \cdots v_m^{e_m})
	 \end{align*}
	 
	 It is easy to verify that the same expression is achieved if we begin with $P(v_1^{e_1} \cdots [w_j,w_j] v_j^{e_j-2} \cdots v_m^{e_m})$, thus $P$ is well defined on monomials of defect zero.
	 
	 Thus, we have given an unambiguous definition of $P(\alpha)$ when either $\alpha$ is of degree $\leq 1$, $\alpha$ has $K$-degree $0$, or $\alpha$ has defect $0$. These three cases with serve as the base cases for induction. In particular, by assuming that $P$ is defined for monomials of degree less than $\alpha$, we have used induction on the degree of the monomial. All that remains is to define $P(\alpha)$ for the case when the defect of $\alpha$ and the $K$-degree of $\alpha$ are both nonzero. To do this, we do induction on both the $K$-degree and the defect. In particular, we do induction on $K$-degree, and for each fixed $K$-degree, we do induction on defect. Thus, we assume that $P$ has already been defined for monomials of smaller $K$-degree than $\alpha$, and furthermore we assume that $P$ has already been defined for monomials with the same $K$-degree as $\alpha$ but lower defect. We now attempt to define $P(\alpha)$ in terms of how $P$ acts on the monomials we have thus already defined it on.
	 
	 In particular, since the defect of $\alpha$ is nonzero, there must be some $i_j$ such that $ i_j > i_{j+1}$, where we recall that $\alpha = v_{i_1}\cdots v_{i_j}v_{i_{j+1}} \cdots v_{i_n}$. For this case, we \emph{define} $P(\alpha)$ to be:

	 $$ P(v_{i_1} \cdots v_{i_{j+1}}v_{i_j} \cdots v_{i_n}) + P(v_{i_1} \cdots dv_{i_{j+1}}dv_{i_j} \cdots v_{i_n}) + P(v_{i_1}\cdots [v_{i_j},v_{i_{j+1}}] \cdots v_{i_n}).$$
	 
	 \ \\
	 The reasons that we may make this definition are:
	 
	 \begin{enumerate}
	 	\item The first summand has already been defined since it has the same $K$-degree but lower defect.
	 	\item The second summand has already been defined since it has lower $K$-degree unless either $v_{i_j}$ or $v_{i_{j+1}}$ is already in $\Ker(d)$, in which case the term is zero anyway).
	 	\item The third summand has already been defined since it has lower (tensor) degree.
	 \end{enumerate}
	 
	 Hence, all that remains is to show that this definition of $P$ is unambiguous; that is, if there are two indices $j,j'$ such that $i_j > i_{j+1}$ and $i_{j'} > i_{j'+1}$, then we must show that the two possible definitions of $P(\alpha)$ given above (the one above and the one obtained by replacing $j$ with $j'$) are equal.
	 
	 There are thus two cases to consider. The first is when the two pairs $(j,j+1)$ and $(j', j'+1)$ do not overlap (w.l.o.g. $j'>j+1$), and second is when the two pairs do overlap (w.l.o.g. $j'=j+1$). 
	 
	 \textbf{Case I:}
	 
	 This is the case where the pairs $(j,j+1)$ and $(j', j'+1)$ do not overlap, where we assume without loss of generality that $j' > j+1$. For convinience, we write $x_1=v_{i_j}$, $x_2 = v_{i_{j+1}}$, $x_3 = v_{i_{j'}}$, and $x_4 = v_{i_{j'+1}}$. Nothing is lost in the following calculation if we replace the monomial $\alpha = v_{i_1} \cdots x_1  x_2 \cdots x_3 x_4 \cdots v_{i_n}$ with the monomial $\alpha' = x_1  x_2  x_3  x_4$. Thus, to show that $P(\alpha')$ is well defined, we must show that the expression:
	 $$ P(x_2x_1x_3x_4) + P(dx_2dx_1x_3x_4) + P([x_1,x_2]x_3x_4)$$
	 is equal to the expression:
	 $$ P(x_1x_2x_4x_3) + P(x_1x_2dx_4dx_3) + P(x_1x_2[x_3,x_4]).$$
	 
	 To show that these expressions are equal, we remark that through the defining properties of $P$, we can rewrite the three summands of the first expression as:
	 
	 \begin{align*}
	 P(x_2x_1x_3x_4)  &= \ P(x_2x_1x_4x_3) &+  P&(x_2x_1dx_4dx_3) &+ P&(x_2x_1[x_3,x_4]) \\ 
	 P(dx_2dx_1x_3x_4) &= \  P(dx_2dx_1x_4x_3) &+ P&(dx_2dx_1dx_4dx_3) &+ P&(dx_2dx_1[x_3,x_4]) \\
	 P([x_1,x_2]x_3x_4) &= \ P([x_1,x_2]x_4x_3) &+ P&([x_1,x_2]dx_4dx_3) &+ P&([x_1,x_2][x_3,x_4]).
	 \end{align*}
	 
	 Likewise, we can rewrite the three summands of the second expression as:
	 
	 \begin{align*}
	 P(x_1x_2x_4x_3)  &= \ P(x_2x_1x_4x_3) &+  P&(dx_2dx_1x_4x_3) &+ P&([x_1,x_2]x_4x_3) \\ 
	 P(x_1x_2dx_4dx_3)  &= \ P(x_2x_1dx_4dx_3) &+  P&(dx_2dx_1dx_4dx_3) &+ P&([x_1,x_2]dx_4dx_3) \\ 
	 P(x_1x_2[x_3,x_4])  &= \ P(x_2x_1[x_3,x_4]) &+  P&(dx_2dx_1[x_3,x_4]) &+ P&([x_1,x_2][x_3,x_4]). \\ 
	 \end{align*}
	 
	 From here, it is easy to see that the two expressions for $P(\alpha')$ must be equal (individual terms can be matched up).
	 
	 \textbf{Case II:}
	 
	 This is the case where the pairs $(j,j+1)$ and $(j',j'+1)$ do overlap, where we assume without loss of generality that $j+1=j'$. For convinience, we write $x_1=v_{i_j}$, $x_2 = v_{i_{j+1}}$, and $x_3 = v_{i_{j+2}}$. Nothing is lost in the following calculation is we replace $\alpha = v_{i_1} \cdots x_1  x_2  x_3 \cdots v_{i_n}$ with the monomial $\alpha' = x_1 x_2 x_3$. Thus, to show that $P(\alpha')$ is well defined, we must show that the expression:
	 $$ A:= P(x_2x_1x_3) + P(dx_2dx_1x_3) + P([x_1,x_2]x_3)$$
	 
	 is equal to the expression:
	 
	 $$B:= P(x_1x_3x_2) + P(x_1dx_3dx_2) + P(x_1[x_2,x_3]).$$
	 
	 We now make some rearragnements to $A$ and $B$, and we note that such rearrangements come straight from the defintion of $P$ and the fact that $P$ is well defined on the terms that we wish to rearrange since they all have either a lower defect, a lower $K$-degree, or a lower tensor degree than $\alpha'$.
	 
	 We first ``rearrange" the first term of $A$ and of $B$ so that they equal $P(x_3x_2x_1)$:
	 
	 \begin{align*}
	 A = P(x_2x_1x_3) &+ P(dx_2dx_1x_3) + P([x_1,x_2]x_3) \\
	 = P(x_2x_3x_1) &+ P(x_2dx_3dx_1) + P(x_2[x_1,x_3]) \\
	 &+ P(dx_2dx_1x_3) + P([x_1,x_2]x_3) \\
	 = P(x_3x_2x_1) &+ P(dx_3dx_2x_1) + P([x_2,x_3]x_1) \\
	 &+ P(x_2dx_3dx_1) + P(x_2[x_1,x_3]) \\
	 &+ P(dx_2dx_1x_3) + P([x_1,x_2]x_3) \\
	 \end{align*}
	 
	 and:
	 
	 \begin{align*}
	 B = P(x_1x_3x_2) &+ P(x_1dx_3dx_2) + P(x_1[x_2,x_3]) \\
	 = P(x_3x_1x_2) &+ P(dx_3dx_1x_2) + P([x_1,x_3]x_2) \\
	 &+ P(x_1dx_3dx_2) + P(x_1[x_2,x_3]) \\
	 = P(x_3x_2x_1) &+ P(x_3dx_2dx_1) + P(x_3[x_1,x_2]) \\
	 &+ P(dx_3dx_1x_2) + P([x_1,x_3]x_2) \\
	 &+ P(x_1dx_3dx_2) + P(x_1[x_2,x_3]).
	 \end{align*}
	 
	 We can now see that the first term of both $A$ and $B$ is $P(x_3x_2x_1)$, so it suffices to show that the second and third columns in the summation expressions for $A$ and $B$ sum to the same value. In order to accomplish this, we will first ``rearrange" the third column of the summation expression for $A$ to get its terms to match up with that of $B$. In particular:
	 
	 \begin{align*}
	 P([x_2,x_3]x_1) = P(x_1[x_2,x_3]) + P(dx_1d[x_2,x_3]) + P([[x_2,x_3],x_1]) \\
	 P(x_2[x_1,x_3]) = P([x_1,x_3]x_2) + P(d[x_1,x_3]dx_2) + P([x_2,[x_1,x_3]]) \\
	 P([x_1,x_2]x_3) = P(x_3[x_1,x_2]) + P(dx_3d[x_1,x_2]) + P([[x_1,x_2],x_3]).
	 \end{align*}
	 
	 Furthermore, we can apply the same process to the second column of $A$ to get the terms to match up with those of $B$, however we note that for this column, the first row of $A$ matches up with the third row of $B$, and vise versa, and furthermore, for each rearrangement, two ``swaps" are necessary.
	 
	 \begin{align*}
	 P(dx_3dx_2x_1) &= P(dx_3x_1dx_2) + P(dx_3[dx_2,x_1]) \\
	 &= P(x_1dx_3dx_2) + P([dx_3,x_1]dx_2) + P(dx_3[dx_2,x_1]) \\
	 P(x_2dx_3dx_1) &= P(dx_3x_2dx_1) + P([x_2,dx_3]dx_1) \\
	 &= P(dx_3dx_1x_2) + P(dx_3[x_2,dx_1]) + P([x_2,dx_3]dx_1) \\
	 P(dx_2dx_1x_3) &= P(dx_2x_3dx_1) + P(dx_2[dx_1,x_3]) \\
	 &= P(x_3dx_2dx_1) + P([dx_2,x_3]dx_1) + P(dx_2[dx_1,x_3]).
	 \end{align*}
	 
	 Hence, we may substitute the previous two series of rearrangements into the summation expressions for $A$ and $B$ and then add the two together (remembering that we are in characteristic $2$) to get:
	 $$ A+B = $$
	 \begin{align*}
	 P([dx_3,x_1]dx_2) + P(dx_3[dx_2,x_1]) + P(dx_1d[x_2,x_3]) + P([[x_2,x_3],x_1]) \\
	 + P(dx_3[x_2,dx_1]) + P([x_2,dx_3]dx_1) + P(d[x_1,x_3]dx_2) + P([x_2,[x_1,x_3]]) \\
	 + P([dx_2,x_3]dx_1) + P(dx_2[dx_1,x_3]) + P(dx_3d[x_1,x_2]) + P([[x_1,x_2],x_3])
	 \end{align*}
	 $$ = $$
	 \begin{align*}
	 P(dx_3d[x_1,x_2]) + P(dx_3[dx_2,x_1]) + P(dx_3[x_2,dx_1]) \\
	 + P(dx_1d[x_2,x_3]) + P([dx_2,x_3]dx_1) + P([x_2,dx_3]dx_1) \\
	 + P(d[x_1,x_3]dx_2) + P([dx_3,x_1]dx_2) + P(dx_2[dx_1,x_3]) \\
	 + P([[x_2,x_3],x_1]) + P([x_2,[x_1,x_3]] + P([[x_1,x_2],x_3]).
	 \end{align*}
	 
	 We now handle each line individually. They are each easily simplified when we remember that $d$ is a derivation over $[,]$.
	 
	 The first line is straightforward:
	 
	 \begin{align*}
	 &P(dx_3d[x_1,x_2]) + P(dx_3[dx_2,x_1]) + P(dx_3[x_2,dx_1]) \\
	 &= P(dx_3(d[x_1,x_2] + [dx_2,x_1] + [x_2,dx_1])) \\
	 &= P(dx_2(d[x_1,x_2] + d[x_2,x_1])) = 0.
	 \end{align*}
	 
	 Here the final equality follows from the fact that $d[x_1,x_2] + d[x_2,x_1] = d([x_1,x_2] + [x_2,x_1]) = d([dx_2,dx_1]) = 0$, which uses the twisted antisymmetry rule $[x,y] + [y,x] + [dy,dx] = 0$.
	 
	 For the second line, we first ``untwist" the first term:
	 $$ P(dx_1d[x_2,x_3]) = P(d[x_2,x_3]dx_1) + P([dx_1,d[x_2,x_3]]).$$
	 
	 Thus we have:
	 \begin{align*}
	 &P(dx_1d[x_2,x_3]) + P([dx_2,x_3]dx_1) + P([x_2,dx_3]dx_1) \\
	 &= P(d[x_2,x_3]dx_1) + P([dx_1,d[x_2,x_3]]) + P([dx_2,x_3]dx_1) + P([x_2,dx_3]dx_1) \\
	 &= P((d[x_2,x_3] + [dx_2,x_3] + [x_2,dx_3])dx_1) + P([dx_1,d[x_2,x_3]]) \\
	 &= P((d[x_2,x_3] + d[x_2,x_3])dx_1) + P([dx_1,d[x_2,x_3]]) \\
	 &= P([dx_1,d[x_2,x_3]]).
	 \end{align*}
	 
	 Finally, for the second line, we need to ``untwist" the third term:
	 
	 $$ P(dx_2[dx_1,x_3]) = P([dx_1,x_3]dx_2) + P([dx_2,[dx_1,x_3]]).$$
	 
	 Thus,
	 
	 \begin{align*}
	 &P(d[x_1,x_3]dx_2) + P([dx_3,x_1]dx_2) + P(dx_2[dx_1,x_3]) \\
	 &= P(d[x_1,x_3]dx_2) + P([dx_3,x_1]dx_2) + P([dx_1,x_3]dx_2) + P([dx_2,[dx_1,x_3]]) \\
	 &= P((d[x_1,x_3] + [dx_3,x_1] + [x_3,dx_1])dx_2) + P([dx_2,[dx_1,x_3]]) \\
	 &= P((d[x_1,x_3] + d[x_3,x_1])dx_2) + P([dx_2,[dx_1,x_3]]) \\
	 &= P([dx_2,[dx_1,x_3]]) .
	 \end{align*}
	 
	 Putting this all together, we get:
	 
	 \begin{align*}
	 A + &B = \\
	 P([dx_1,d[x_2,x_3]]) &+ P([dx_2,[dx_1,x_3]])  \\
	 + P([[x_2,x_3],x_1]) + P([x_2,[&x_1,x_3]]) + P([[x_1,x_2],x_3]).
	 \end{align*}
	 
	 We note, however, that the twisted antisymmetry rule tells us that $[[x_2,x_3],x_1] = [x_1,[x_2,x_3]] + [dx_1,d[x_2,x_3]]$. Appying $P$ to this equation and substituting the result in gives us:
	 
	 \begin{align*}
	 A + &B = \\
	 P([dx_1,d[x_2,x_3]]) &+ P([dx_2,[dx_1,x_3]])  \\
	 + P([x_1,[x_2,x_3]]) + P([dx_1,d[x_2,x_3]]) &+ P([x_2,[x_1,x_3]]) + P([[x_1,x_2],x_3]) \\
	 = P([x_1,[x_2,x_3]]) + P([x_2,[x_1,x_3]]) +& P([dx_2,[dx_1,x_3]]) + P([[x_1,x_2],x_3]) \\
	 = P([x_1,[x_2,x_3]] + [x_2,[x_1,x_3]] +& [dx_2,[dx_1,x_3]] + [[x_1,x_2],x_3]).
	 \end{align*}
	 
	 Finally, we can recognize that if we substitute $x=x_1$, $y=x_2$, and $z=x_3$ into the twisted jacobi identity at the start, we get:
	 $$ [x_1,[x_2,x_3]] + [x_2,[x_1,x_3]] + [dx_2,[dx_1,x_3]] + [[x_1,x_2],x_3] = 0.$$
	 
	 If we apply $P$ to this identity, we get $A+B=0$, and consequently $A=B$. Thus, $P$ is well defined for this case as well.
	 
	 \section{Acknowledgements}
	 
	 I would first and foremost like to thank the MIT PRIMES program for facilitating this research project. In addition, I would like to thank my mentor Lucas Mason-Brown for countless invaluable meetings and guidance in my research, as well as Professor Pavel Etingof for suggesting the project and providing advice and help throughout the project (Theorems 2.23 and 2.26 are effectively due to him). Finally, I would like to thank Dr. Tanya Khovanova for her feedback on the drafts of this paper. This project could not have happened without the contributions of each of the aforementioned people. 
	 
	 \section{Bibliography}
	 
	 [1] Dummit, David Steven., and Richard Foote M. Abstract Algebra. Abstract Algebra. Hoboken,
	 NJ: Wiley, 2004.
	 
	 [2] Etingof, Pavel. Koszul duality and the PBW theorem in symmetric tensor categories in positive characteristic. Nov 2016. URL \url{https://arxiv.org/abs/1603.08133}
	 
	 [3] Etingof, P. I., Shlomo Gelaki, Dmitri Nikshych, and Victor Ostrik. Tensor categories. Mathematical Surveys and Monographs, 205, 2015. URL \url{http://www-math.mit.edu/~etingof/egnobookfinal.pdf}.
	 
	 [4] Mitchell, Steve. The Jacobson Radical. 04 2015. URL \url{https://sites.math.washington.edu/~mitchell/Algh/jac.pdf}.
	 
	 [5] Venkatesh, Siddharth. Hilbert Basis Theorem and Finite Generation of Invariants in Symmetric
	 Fusion Categories in Positive Characteristic. July 2015. URL \url{https://arxiv.org/abs/1507.05142}.

\end{document}